\documentclass{amsart}
\usepackage{graphicx} % Required for inserting images

\usepackage{amsmath}
\usepackage{amsfonts}
\usepackage{amsthm}
\usepackage{amssymb}
\usepackage{epsfig}
\usepackage{float, color}
\usepackage{xcolor}

\newtheorem{thm}{Theorem}

\newtheorem{lem}{Lemma}
\newtheorem{prop}{Proposition}

\theoremstyle{definition}
\newtheorem{definition}{Definition}%[chapter]

\theoremstyle{remark}
\newtheorem{remark}{Remark}%[chapter]

\theoremstyle{remark}
\newtheoremstyle{named}{}{}{\itshape}{}{\bfseries}{.}{.5em}{\thmnote{#3}#1}
\theoremstyle{named}
\newtheorem*{namedtheorem}{}
\long\def\symbolfootnote[#1]#2{\begingroup\def\thefootnote{\fnsymbol{footnote}}
\footnote[#1]{#2}\endgroup}

\theoremstyle{definition}

\newcommand{\CC}{\mathbb{C}}

\newcommand{\QQ}{\mathbb{Q}}

\newcommand{\ZZ}{\mathbb{Z}}

\newcommand{\PGL}{\operatorname{PGL}}

\newcommand{\Manoa}{M\=anoa}
\newcommand{\Hawaii}{Hawai\kern.05em`\kern.05em\relax i}

\title{On a slice of the cubic 2-adic Mandelbrot set}
\author{Jacqueline Anderson}
\email{jacqueline.anderson@bridgew.edu}
\address{Mathematics Department, Bridgewater State University, Bridgewater, MA 02325 USA}
\author{Emerald Stacy}
\email{estacy2@washcoll.edu}
\address{Department of Mathematics and Computer Science, Washington College, Chestertown, MD 21620 USA}
\author{Bella Tobin}
\email{btobin@agnesscott.edu}
\address{Department of Mathematics, Agnes Scott College, Decatur, GA 30030 USA}

\begin{document}

\maketitle
\begin{abstract}
    Consider the one-parameter family of cubic polynomials defined by $f_t(z) =-\frac 32 t(-2z^3+3z^2)+1, t \in \mathbb{C}_2$. This family corresponds to a slice of the parameter space of cubic polynomials in $\mathbb{C}_2[z]$. We investigate which parameters in this family belong to the cubic $2$-adic Mandelbrot set, a $p$-adic analog of the classical Mandelbrot set. When $t=1$, $f_t(z)$ is post-critically finite with a strictly preperiodic critical orbit. We establish that this is a non-isolated boundary point on the cubic $2$-adic Mandelbrot set and show asymptotic self-similarity of the Mandelbrot set near this point. Subsequently, we investigate the Julia set for polynomial on the boundary and demonstrate a similarity between the Mandelbrot set at this point and the Julia set, similar to what is seen in the classical complex case. 
\end{abstract}
\section{Introduction}
The \emph{Mandelbrot} set, perhaps the most famous set in complex dynamics, is a subset of the complex plane determined by the dynamical properties of quadratic polynomials of the form $f_c(z)=z^2+c$. The Mandelbrot set may be defined in terms of the forward orbit of the critical point $0$ under iteration of $f_c$. The (forward) orbit of a point $z$ under iteration of a map $f$ is denoted as follows:
\[ \mathcal{O}_f(z) = \{ f^n(z): n \in \mathbb{Z}_{ \geq 0} \}, \] 
where $f^n = f \circ f \circ \dots \circ f$ denotes the $n$-fold composition of $f$ with itself.
The Mandelbrot set is then defined as 
\[\mathcal{M} = \lbrace c\in \CC : \mathcal{O}_{f_c}(0)\text{ is bounded for }\ f_c(z) = z^2+c\rbrace.\]

A \emph{Misiurewicz point} is a point in the Mandelbrot set for which the orbit of 0 is strictly preperiodic. These points all appear on the boundary of the Mandelbrot set, and in fact are dense in the boundary of $\mathcal{M}$. In 1989, Tan Lei proved that the Mandelbrot set zoomed in on a Misiurewicz point $c$ and the corresponding Julia set for $f_c$ zoomed in on $c$ at the same scale were asymptotically similar about $c$ \cite{TanLei}. Moreover, both are asymptotically self-similar as well. The aim of this paper is to explore the extent to which similar properties exist in the $2$-adic setting for cubic polynomials. 

The Mandelbrot set is of great interest in complex dynamics and has been an inspiration for analogues in the $p$-adic setting. One might notice that looking at the Mandelbrot set in the $p$-adic setting yields less interesting results than in the complex setting. In particular, the Mandelbrot set for the family of maps $z^2+c$ is the unit disk in $\CC_p$. A natural extension of the Mandelbrot set is to look at polynomials of higher degrees.

We say a map $f$ is \emph{post-critically bounded}, or PCB, if all of its critical points have bounded forward orbits under iteration of $f$. Similarly, we say $f$ is \emph{post-critically finite}, or PCF, if all of its critical points have finite forward orbits. 

Let $\mathcal{P}_{d,p}$ denote the $(d-1)$-dimensional parameter space of degree $d$ polynomials $f$ defined over $\mathbb{C}_p$, normalized so that $f$ is monic and $f(0)=0$. (Note that every polynomial can be conjugated so that it is in this normal form, but that representation is not unique in that there exist different polynomials in $\mathcal{P}_{d,p}$ that are conjugates of each other.) We can define the degree $d$ $p$-adic Mandelbrot set, $\mathcal{M}_{d,p}$, as the set of parameters in $\mathcal{P}_{d,p}$ that yield a PCB polynomial. When $p \geq d$, this generalized $p$-adic Mandelbrot set is similar to what we see in the quadratic case: a polynomial in $\mathcal{P}_{d,p}$ is PCB if and only if all of its coefficients belong to the unit disk in $\mathbb{C}_p$.

When $p<d$, however, the $p$-adic Mandelbrot set $\mathcal{M}_{d,p}$ has an interesting structure more akin to that of the complex Mandelbrot set; in particular, this set has boundary points. In \cite{anderson2010p} the first author found a point on the boundary of $\mathcal{M}_{3,2}$. She completed this by investigating the one-parameter family of cubic polynomials of the form $g_t(z) = z^3-\frac{3}{2}tz^2$ for $t \in \mathbb{C}_2$. Note that $g_t$ has two critical points, $0$ and $t$, and that $0$ is fixed. When $t=1$ it is easy to see that both critical points have a finite forward orbit. In \cite{anderson2010p} it was shown that $g_1(z) = z^3-\frac{3}{2}z^2$ is a boundary point on the cubic $2$-adic Mandelbrot set.

In this paper we explore a different slice of the cubic $2$-adic Mandelbrot set, the one-parameter family defined as follows:
\[ f_t(z) = -\tfrac{3}{2} t (-2z^3+3z^2)+1. \]

The polynomials in this family have critical points $0$ and $1$, and $f_t(0) = 1$, so there is just one critical orbit to keep track of. Note that we are using a different normal form than the one used in the definition of $\mathcal{P}_{d,p}$, as these polynomials are not monic and zero is not fixed. Rather, in this family we choose to normalize so the critical points are $0$ and $1$, following the normal form for bicritical polynomials as used in \cite{anderson2020cubic, TobinThesis}. Using this family, we can find another point on the boundary of $\mathcal{M}_{3,2}$.

\begin{namedtheorem}[Main Theorem]\label{thm:main}
The point associated to $f_1(z) = -\frac{3}{2}(-2z^3+3z^2)+1$ is on the boundary of the cubic $2$-adic Mandelbrot set. 
\end{namedtheorem}

In the complex quadratic setting, Tan Lei showed that there are similarities between the Mandelbrot set near a Misiurewicz point and the Julia set for the corresponding quadratic polynomial. Inspired by this result, we explore the Julia set of $f_1(z)$ in Section \ref{sec:julia} to see if there is any such relationship in the $p$-adic setting for the boundary point given in the Main Theorem.

\subsection{Outline} We begin by providing background and necessary definitions in Section \ref{sec:prelim}. In Section \ref{sec: mandelset} we prove the Main Theorem by identifying one sequence of disks approaching $t=1$ in which the parameters belonging to those disks are not in the Mandelbrot set, and other sequences of disks approaching $t=1$ whose contents are in the Mandelbrot set. Then, in Section \ref{sec:julia} we discuss the $2$-adic Julia set for $f_1(t) = -\frac{3}{2}(-2z^3+3z^2)+1$, demonstrating that there appears to be similarities between the Mandelbrot set and this Julia set. Those similarities can be seen when one compares the results in Theorems~\ref{thm: unbdd} and~\ref{thm:q2bdd} on the Mandelbrot set to those in Propositions~\ref{prop:unbddjulia} and~\ref{prop:bddjulia} on the Julia set for $f_1$.
\subsection{Acknowledgements}
This work was initiated during the Women in Numbers 6 workshop at Banff International Research Station (BIRS) in 2023. The authors thank BIRS and the WIN6 organizers for providing accommodations that allowed us to focus on this work and the opportunity to collaborate on this project.
\section{Preliminaries}\label{sec:prelim}
\subsection{Basic definitions in arithmetic dynamics} As we are interested in the behavior of points under iteration of rational maps, we begin this section by defining terms for the possible behavior. Let $K$ be a field and $f(z) \in K(z)$. We say a point $\alpha$ is \emph{periodic} for $f$ if there exists some positive integer $n$ such that $f^n(\alpha) = \alpha$. Note that if $n= 1$ then $\alpha$ is a \emph{fixed point} for $f$. If $n$ is the smallest positive integer such that $f^n(\alpha) = \alpha$ then $\alpha$ is periodic of \emph{exact period $n$}.  If there exist distinct positive integers $n,m$ such that $f^n(\alpha) = f^m (\alpha)$, we say $\alpha$ is a \emph{preperiodic point} for $f$. If $\alpha$ is preperiodic but not periodic, then we say it is \emph{strictly preperiodic.} If $\alpha$ is not periodic or preperiodic, so it has an infinite forward orbit under $f$, we say $\alpha$ is a \emph{wandering point}. 

Periodic points can have an effect on the dynamical behavior of nearby points.

\begin{definition}
    Let $\alpha$ be a point of exact period $n$, so $f^n(\alpha) = \alpha$ and $f^m(\alpha)\neq \alpha$ for any $0<m<n$. Then the \emph{multiplier} of $\alpha$ for $f$ is 
    \[
    \lambda_f(\alpha) := (f^n)'(\alpha).
    \]
\end{definition}
The multiplier is invariant under iteration, so all points in the orbit of a periodic point will have the same multiplier. 

We say that $\alpha$ is repelling for $f$ if $|\lambda_f(\alpha)|>1$, attracting if $|\lambda_f(\alpha)|<1$, and neutral if $|\lambda_f(\alpha)|=1$. One can think of the multiplier of $\alpha$ as an indicator of the behavior of nearby points under iteration of $f$. For example, if $\alpha$ is a repelling periodic point, then there exists a neighborhood of $\alpha$ in which the points move away from $\alpha$ under iteration of $f$, and if $\alpha$ is attracting, then there exists a neighborhood of $\alpha$ in which the points move closer to $\alpha$ with each iteration of $f$. 

We say that $f(z) \in K(z)$ is \emph{conjugate} to a rational map $g$ if there exists $\varphi(z) \in \PGL_2(\bar{K})$ such that 
\[g = f^\varphi:= \varphi^{-1}\circ f \circ \varphi.\]

Conjugation commutes with iteration so if $\alpha$ is a periodic point for $f$ then $\varphi^{-1}(\alpha)$ is periodic point (of the same period) for $f^\varphi$. Similarly, critical points for $f$ correspond to critical points for $f^\varphi$. In particular, if $f$ is a PCF (or PCB) polynomial than $f^\varphi$ is PCF (resp. PCB) for any $\varphi\in\PGL_2(\bar{K})$. In this way conjugation respects many dynamical properties and allows us to investigate the dynamics of a map by choosing a convenient representative from its conjugacy class. For example, every quadratic polynomial is conjugate to exactly one polynomial of the form $z^2+c$, which is why polynomials of that form are used in the definition of the classical Mandelbrot set. 

\subsection{Julia sets}

For a rational map $f\in \CC(z)$, the Julia set of $f$ has many equivalent definitions. 
Among others, the \emph{Julia set} of $f$ is characterized as 
\begin{itemize}
    \item The closure of the repelling periodic points.
    \item The smallest closed set, $\mathcal{J}$, in $\CC$ containing at least 3 points that is completely invariant under $f$. That is, if $f(\alpha) \in \mathcal{J}$ and $f^{-1}(\alpha) \in \mathcal{J}$ for any $\alpha \in \mathcal{J}$.

\end{itemize}

For a polynomial $f$, the \emph{filled Julia set} is the set of points with bounded forward orbits, and the Julia set is the boundary of the filled Julia set. Thus, a point is in the Julia set if its forward orbit is bounded but there are arbitrarily close points whose orbits tend to infinity.
The parameters $c$ that belong to the complex Mandelbrot set are also precisely those for which the Julia set of $f_c$ is connected. 

\subsection{Notation}

Throughout the paper we will write $|\cdot |$ to denote the normalized $2$-adic absolute value and $\overline{D}(a,\delta)$ to denote the closed disk in $\mathbb{C}_2$ centered at $a$ of radius $\delta$. We recall that $z \in \overline{D}(a,2^{-r})$ if and only if $z \equiv a \mod {2^{r}}$.

\section{A point on the cubic $2$-adic Mandelbrot Set}\label{sec: mandelset}

 In this section we analyze the parameter space associated to the one-parameter family of cubic polynomials of the form $f_t(z) = -\frac32 t (-2z^3+3z^2)+1$ in $\mathbb{C}_2[z]$. Recall that in this family, the critical points are $0$ and $1$, and $f_t(0)=1$, so there is only one critical orbit to consider. We are concerned with which parameters $t$ belong to the Mandelbrot set for this family; in other words, for which $t$ values is the critical orbit bounded? We first note that if $|t|>1$, the critical orbit is always unbounded, as $|f_t(1)| = |1-\frac32 t| = 2|t|>2$, and all points of absolute value greater than 2 have unbounded orbits. If $|t| \leq \frac12$, on the other hand, the coefficients of $f_t$ are all inside the $2$-adic unit disk, and so all points in the unit disk stay bounded under iteration of $f_t$. So, all $t$ with $|t| \leq \frac12$ belong to the Mandelbrot set for this family, while all $t$ with $|t| >1$ do not belong to the Mandelbrot set. We  examine the case where $|t|=1$.
                        
 The proof of the main theorem lies in the statements of Theorems \ref{thm: unbdd} and \ref{thm: bdd} below. Theorem \ref{thm: unbdd} establishes that there are values of $t$ arbitrarily close to $t=1$ such that the map $f_t(z)$ is post-critically bounded, while Theorem \ref{thm: bdd} establishes the reverse: that there are values of $t$ arbitrarily close to $t=1$ for which the map $f_t(z)$ is post-critically unbounded. In Figure \ref{fig:ourMandelbrotset} below, we show the structure of the slice of the $2$-adic, degree $3$ Mandelbrot set for this family. Each node corresponds to a disk of $t$ values centered at the given value with radius $2^{-r}$, where $r$ is indicated by the power of 2 in the modulus in the right column. The nodes are ovals if the critical orbit of the corresponding polynomials are known to be unbounded, and squares if they are known to be bounded. 

\begin{figure}[h]
\centering
\includegraphics[width=0.8\textwidth]{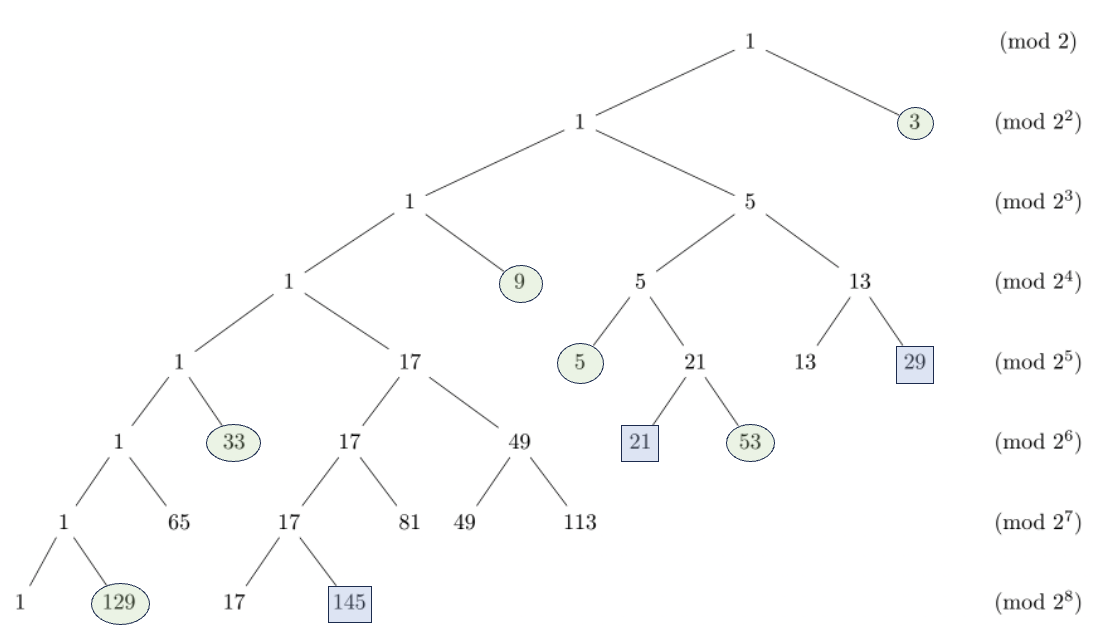}
\caption{} 
\label{fig:ourMandelbrotset}

\end{figure}

We begin with a lemma in which we build on the fact that $-\frac{1}{2}$ is a repelling fixed point for $f_1(z)$, which implies that points sufficiently close to $-\frac{1}{2}$ move away from $-\frac{1}{2}$ by a fixed amount with each iteration of $f_1$. For parameters $t$ sufficiently close to $1$, we are able to show that the distance between certain points and $-\frac{1}{2}$ increases by a factor of $4$ upon each iteration of $f_t$.

\begin{lem} \label{lemma1}
Suppose that $t \in \overline{D}\left(1+2^{2n-1},2^{-2n}\right)$ for some $n \geq 1$ and assume $|z+\frac{1}{2}|=2^k$, where $-2n<k<1$. Then $|f_t(z)+\frac{1}{2}|=2^{k+2}$. 
\end{lem}
\begin{proof}
First, note that $f_t(z)+\frac{1}{2} = (z+\frac{1}{2})(3tz^2-6tz+3t)-\frac{3}{2}(t-1)$. This gives us that 
\begin{align*}
\left|f_t(z)+\tfrac{1}{2}\right| &\leq \max \lbrace |z+\tfrac{1}{2}|\cdot |3t| \cdot |z^2-2z+1|, |\tfrac{3}{2}|\cdot |t-1|\rbrace\\
&= \max \lbrace 2^k|z-1|^2, 2\cdot 2^{-(2n-1)}\rbrace,
\end{align*}
with equality if the values for which we are taking the maximum of are distinct. 
Since $|z+\tfrac{1}{2}| = 2^k<2$, it must be the case that $|z|=2$ and therefore $|z-1|=2$, so we have that 
\[
\max \lbrace 2^k|z-1|^2, 2\cdot 2^{-(2n-1)}\rbrace = \max \lbrace 2^{k+2}, 2^{-2n+2}\rbrace = 2^{k+2}>2^{-2n+2}.
\]

Thus, by the strong triangle inequality, we have that $\left|f_t(z)+\frac{1}{2}\right| = 2^{k+2}$, as desired.  
\end{proof}

    In Theorem \ref{thm: unbdd} we establish a repeating pattern for $2$-adic disks arbitrarily close to $t=1$ yielding post-critically unbounded polynomials. This pattern is highlighted in Figure \ref{fig:repeatedmandelset} with the highlighted disks representing discs of $t$-values for which $f_t$ is post-critically unbounded.

\begin{figure}[h]
    \centering
    \includegraphics[width=0.8\textwidth]{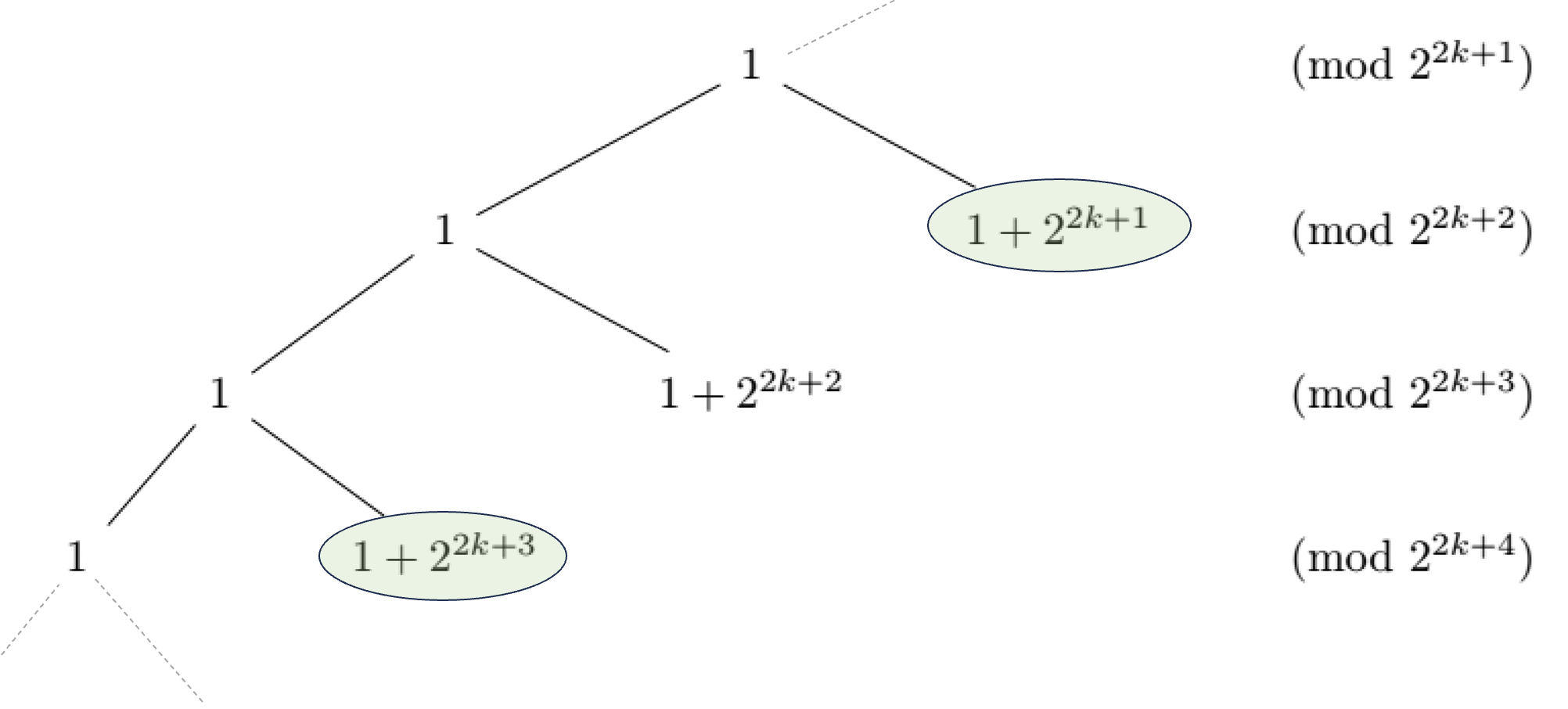}
    \caption{}
    \label{fig:repeatedmandelset}
\end{figure}

\begin{thm}\label{thm: unbdd}

If $t \in \overline{D}(1+2^{2n-1}, 2^{-2n})$ for $n\geq 1$ then $f_t(z) = -\frac{3}{2} t (-2z^3+3z^2)+1$ is post-critically unbounded. 

\end{thm}

\begin{proof}
Since $f_t(1) +\frac{1}{2} = \frac{3}{2}(1-t)$ then for $t \in \overline{D}(1+2^{2n-1}, 2^{-2n})$ we have that 
\[ |f_t(1) + \tfrac12| = 2|1-t| = 2^{1-(2n-1)} = 2^{-2n+2}.\]
Then, repeatedly applying Lemma~\ref{lemma1} shows the following (for $m \leq n+1$):
\[ \left|f_t^m(1) + \tfrac12\right| = 2^{-2n+2m}.\]
In particular, $|f_t^n(1)+\frac12| = 1$, and $|f_t^{n+1}(1)+\frac12| = 4$. This implies that $|f^{n+1}(1)| = 4>2$. Any point of absolute value greater than 2 will have an orbit that escapes to infinity, and so the orbit of the critical point $1$ is unbounded.
\end{proof}

The previous theorem gives a sequence of disks getting arbitrarily close to $t=1$ whose elements all correspond to polynomials with unbounded critical orbits. The next theorem shows that there is a different sequence of disks, also getting arbitrarily close to $t=1$, whose elements all correspond to PCB polynomials. This establishes $t=1$ as a boundary point (and not an isolated point) in the Mandelbrot set for this family.

\begin{thm}\label{thm: bdd} There exists a sequence $\{t_n\}_{n=3}^\infty$ such that $\displaystyle\lim_{n \to \infty} t_n = 1$ and $f_{t_n}$ is PCF with a periodic critical orbit of length $n$. Moreover, for $t\in \overline{D}(t_n,2^{8n/3+2}$ the polynomial $f_t(z)$ is PCB. 
\end{thm}

\begin{proof} 
Consider the polynomial in $s$ defined by $g_n(s) = f_{s+1}^n(0)$. The roots of $g_n$ are parameters $s$ for which 0 is periodic of period $n$ under iteration of $f_{s+1}$. We'll show that the Newton polygon for $g_n$ has vertices at $(0,-1)$ and $(1,3-2n)$, which implies that there exists a root $s_n$ of $g_n$ of absolute value $2^{4-2n}$. Setting $t_n=s_n+1$, we have a sequence of $t$ values approaching 1 for which the critical orbit is periodic, and therefore bounded.

We will prove this by induction. First, take $n=3$. In this case, we can explicitly compute $g_3(s)$:
\[ g_3(s) = -\tfrac12 - \tfrac{93}{8}s-\tfrac{243}{8}s^2-\tfrac{243}{8}s^3-\tfrac{81}{8}s^4.\]
The leftmost segment of the Newton polygon for this polynomial has slope $-2$, as desired. This implies that $g_3$ has a root $s_3$ of absolute value $\frac14$. Take $t_3 = 1+s_3$.

For the inductive step, suppose that $g_n(s)$ is of the following form:
\[ g_n(s) = -\tfrac12+ \sum_{i=1}^{3n-5} a_i s^i,\]
where $|a_1| = 2^{2n-3}$ and $|a_i| \leq 2^{1-(4-2n)i}$ for $i>1$. This is what is necessary to produce the desired Newton polygon properties. We must show that $g_{n+1}(s)$ is of the form
\[ g_{n+1}(s) = -\tfrac12 + \sum_{i=1}^{3n-2} b_i s^i,\]
where $|b_1| = 2^{2n-1}$ and $|b_i| \leq 2^{1-(2-2n)i}$ for $i>1$. 

We compute $g_{n+1}(s)$ from $g_n$:
\[ g_{n+1}(s) = -\tfrac32(s+1)\left( -\tfrac12 + \sum_{i=1}^{3n-5} a_i s^i \right)^2 \left(3-2\left( -\tfrac12 + \sum_{i=1}^{3n-5} a_i s^i \right) \right) + 1
\]

\[ =-\tfrac32(s+1)\left( -\tfrac12 + \sum_{i=1}^{3n-5} a_i s^i \right)^2 \left(4 - \sum_{i=1}^{3n-5} 2a_i s^i \right) + 1  \] 
\[= \left(-3s-3\right) \left( -\tfrac12 + \sum_{i=1}^{3n-5} a_i s^i \right)^2 \left(2 - \sum_{i=1}^{3n-5} a_i s^i \right) +1
\]

By expanding the above, we see that the constant term is $-\frac12$, and the degree 1 term is 
\[b_1 = \tfrac{27}{4}a_1 -\tfrac32,\]
which has the desired absolute value.

It remains to argue that for $2 \leq m \leq 3n-2$, the coefficient $b_m$ has absolute value less than or equal to $2^{1-(2-2n)m}$. Carefully considering the expansion of the expression for $g_{n+1}(s)$ above, we see that the terms of degree $m$ have coefficients of the form $c a_i a_j a_k$, where $i+j+k=m$ or $i+j+k=m-1$ and $|c| \leq 4$. This implies
\[ |b_m| \leq 4 \cdot 2^{1-(4-2n)i} \cdot2^{1-(4-2n)j} \cdot2^{1-(4-2n)k} \leq 2^{5-(4-2n)m} \leq 2^{1-(2-2n)m},
\]
as desired. 

Now that we have established the existence of a sequence $\lbrace t_n\rbrace_{n=3}^\infty$ with $\lim\limits_{n\rightarrow \infty}t_n= 1$ such that $f_{t_n}$ has a periodic critical orbit of length $n$, we will show that values of $t$ in disks centered at $t_n$ will correspond to PCB polynomials.  
Specifically, by tracking the precision of a small disk around $t_n$ as we iterate, we see that if $|t-t_n| \leq 2^{-(\frac83 n+2)}$, then the critical orbit for $f_t$ will be bounded.

Write $t=t_n+2^m v$, where $m \geq \frac83 n + 2$ and $|v| \leq 1$. Let $\{0, 1, c_2, c_3, \dots, c_{n-1}\}$ denote the periodic critical orbit for $f_{t_n}$. Letting $k$ be the integer nearest to $\frac23 n$, we will show that $f_t^n$ will map a disk centered at 0 of radius $2^{-k}$ into itself, and thus the critical orbit is bounded. If $z \in \bar{D}(0,2^{-k})$, then $z=2^ku$ for some $u$ with $|u| \leq 1$. Then,
\[ f_t(z) = 1 -\tfrac32 (t_n +2^m v)(z^2)(3-2z) = 1-\tfrac32 (t_n + 2^m v)(2^{2k}u^2)(3-2^{k+1}u) \] 
\[\equiv 1 \pmod{2^{2k-1}}. \]

So $f_t(\bar{D}(0,2^{-k})) \subset \bar{D}(1,2^{-(2k-1)})$.

Looking at the next iterate, we see that $f_t(\bar{D}(1,2^{-(2k-1)})) \subset \bar{D}(c_2, 2^{-(4k-1)})$ by taking $z=1+2^{2k-1}u$ for some $u$ such that $|u| \leq 1$ and computing $f_t(z)$:
\begin{align*}
    f_t(1+2^{2k-1}u) &= 1-\tfrac32(t_n+2^m v)(1+2^{2k}u+2^{4k-2}u^2)(1-2^{2k}u)\\
    &= 1-\tfrac32(t_n+2^m v)(1-2^{4k}u^2+2^{4k-2}u^2-2^{6k-2}u^3)\\
    &\equiv 1- \tfrac32 t_n \pmod{2^{4k-1}}.
\end{align*}

Since $c_2 = f_t(1) =1-\frac{3}{2}t_n$, we have that $f_t(1+2^{2k-1}u) \equiv c_2 \pmod {2^{4k-1}}$.

Continuing to track iterates from here, the radius of the disk will expand by a factor of 4 each time. We can see this by showing that if $|z-z_n| \leq 2^{-r}$, $|z| \leq 2$ and $r \leq m-1 $ then $|f_t(z) - f_{t_n}(z_n)| \leq 2^{-(r-2)}$. Write $z=z_n + 2^r u$ for some $u$ such that $|u| \leq 1$. Then, 
\[ f_t(z_n + 2^r u) = 1-\tfrac32(t_n+2^m v)(z_n^2+2^{r+1}z_n u+2^{2r}u^2)(3-2z_n-2^{r+1}u) \]
\[ \equiv 1-\tfrac32 t_n (z_n^2)(3-2 z_n)  \equiv f_{t_n}(z_n) \pmod{2^{r-2}} .\]

Applying this $n-2$ times to $D(c_2,2^{-(4k-1)})$ gives the following:
\[ f_t^n(\bar{D}(0,2^{-k})) \subset \bar{D}(0, 2^{-(4k-1-2(n-2))}).\]
Note that $4k-2n+3 >k$, because this is equivalent to $k>\frac23 n - 1$, and $k$ was chosen to be the nearest integer to $\frac23 n$. So, $f_t^n(\bar{D}(0,2^{-k})) \subset \bar{D}(0,2^{-k})$, and the proof is complete.
\end{proof}

 In addition to the sequence of disks belonging to the Mandelbrot set described in the previous theorem, in Theorem~\ref{thm:q2bdd} we identify two sequences of disks approaching $t=1$ in which all $\mathbb{Q}_2$ points belong to the Mandelbrot set.

\begin{thm}\label{thm:q2bdd} For all integers $n \geq 3$, define the following two sequences of disks:
\[ D_{a,n} = \overline{D}\left(1+5 \cdot 2^{2n}, 2^{-(2n+3)}\right)\]
\[ D_{b,n} = \overline{D} \left(1+7 \cdot 2^{2n}, 2^{-(2n+3)}\right).\]
Then any $t \in \mathbb{Q}_2$ belonging to one of these disks belongs to the Mandelbrot set.
\end{thm}
Figure \ref{fig:repeatedmandelset2} shows the self-similar pattern of disks in the slice of the cubic Mandelbrot set described in Theorems \ref{thm: unbdd} and \ref{thm:q2bdd}. The shaded ovals represent disks of $t$-values for which $f_t$ is post-critically unbounded, as in Theorem \ref{thm: unbdd}. The rectangles represent disks whose $\QQ_2$ points described in Theorem \ref{thm:q2bdd} correspond to $t$-values for which $f_t$ is post-critically bounded, and thus are in the cubic $2$-adic Mandelbrot set. 
\begin{figure}[h]
    \centering
    \includegraphics[width=0.8\textwidth]{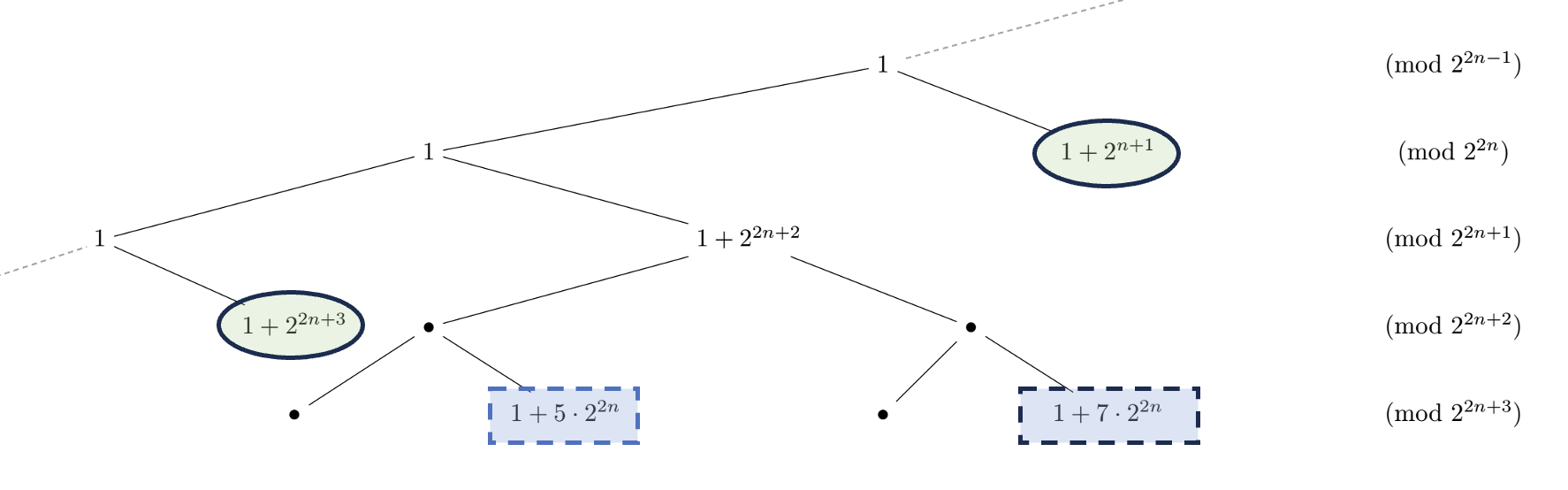}
    \caption{}
    \label{fig:repeatedmandelset2}
\end{figure}

\begin{remark} In Theorem~\ref{thm:q2bdd} above, we must restrict to $t \in \mathbb{Q}_2$ to guarantee that the critical orbit for $f_t$ is bounded. This is because for these $t$ values, the critical orbit eventually ends up near a repelling 2-cycle. Since that cycle is repelling, there are arbitrarily close points whose orbits escape to infinity, but it can be shown that all $\mathbb{Q}_2$ points near that repelling 2-cycle have bounded orbits.
\end{remark}

To prove Theorem~\ref{thm:q2bdd}, we will need a few lemmas. The first establishes that for $t$ values sufficiently close to 1, a $\mathbb{Q}_2$ point whose orbit eventually enters one of two particular disks will have a bounded orbit under iteration of $f_t$.

\begin{lem} \label{lem:2cycle}
Let $t \in \mathbb{Q}_2, t \equiv 1 \pmod{32}$. Then if $z \in \mathbb{Q}_2$ and either $z \in \overline{D}\left(\frac{19}{2}, 2^{-4}\right)$ or $z \in \overline{D}\left(\frac{27}{2},2^{-4}\right)$, $\mathcal{O}_{f_t}(z)$ is bounded.
\end{lem}

\begin{proof}
The proof is a straightforward computation of iteration of disks. First we show that a point in $\overline{D}\left(\frac{27}{2}, 2^{-4}\right)$ maps into $\overline{D}\left(\frac{19}{2}, 2^{-4}\right)$ after three iterations. Let $z=\frac{27}{2}+16k_0$, where $k_0 \in \mathbb{Z}_2$. Then applying $f_t$ to $z$, we get
\[ f_t\left(\frac{27}{2} + 16k_0\right) = 1+6561t + 24300k_0t + 29952k_0^2t + 12288k_0^3t \in \overline{D}\left(2,2^{-2}\right).\]
Next, we apply $f_t$ to a point of the form $2+4k_1, k_1 \in \mathbb{Z}_2$:
\[ f_t\left(2+4k_1\right) = 1+6t + 72k_1t+216k_1^2t + 192k_1^3t \in \overline{D}\left(3,2^{-2}\right).\]
Finally, we apply $f_t$ to a point of the form $3+4k_2, k_2 \in \mathbb{Z}_2$:
\[ f_t\left(3+4k_2\right)  =  1+\frac{81}{2} t + 216 k_2t+360k_2^2t + 192 k_2^3t \equiv \frac{19}{2} + 8k_2(k_2+1) \pmod{16}.\]
Note that since $k_2 \in \mathbb{Z}_2$, we know $k_2(k_2+1) \equiv 0 \pmod{2}$, and so this point lies in the disk $\overline{D}\left(\frac{19}{2},2^{-4}\right)$.

Next, we show that points in $\overline{D}\left(\frac{19}{2},2^{-4}\right)$ map back into the disk $\overline{D}\left(3,2^{-2}\right)$, and so $\mathbb{Q}_2$ points in these disks have bounded orbits under iteration of $f_t$.

Let $z=\frac{19}{2} + 16k_3, k_3 \in \mathbb{Z}_2$. Then we have
\[f_t\left( z \right) = 1+2166t+11628k_3t+20736k_3^2t+12288k_3^3t  \in \overline{D}\left(3,2^{-2}\right),\]
as desired.
\end{proof}

In the next lemma, we show that if $t$ and $t_0$ are sufficiently close together, then $f_t$ and $f_{t_0}$ will map disks of a certain type to the same image.

\begin{lem}\label{lem:tvst0}
Let $\alpha \in \mathbb{C}_2$ such that $\left|\alpha-\frac12\right| \leq 1$ and let $r \leq 1$. Suppose that $f_{t_0}$ maps $\overline{D}\left(\alpha,r\right)$ to $\overline{D}\left(\beta,4r\right)$, where $\beta=f_{t_0}(\alpha)$. Then if $\left| t-t_0 \right| \leq r$, $f_t$ will also map $\overline{D}\left(\alpha,r\right)$ to $\overline{D}\left(\beta,4r\right)$.
\end{lem}

\begin{proof}
Write $t=t_0+\epsilon$, where $| \epsilon | \leq r$. Let $z=\alpha+\delta$, where $|\delta | \leq r$. Then we have
\[f_t(z) =  -\tfrac32(t_0+\epsilon)(-2z^3+3z^2)+1 = f_{t_0}(z) - \tfrac32 \epsilon(-2z^3+3z^2).   \]
This will lie in $\overline{D}\left(\beta,4r\right)$ if $\left| \tfrac32 \epsilon(-2z^3+3z^2)\right| \leq 4r$. Since $z \equiv \tfrac12 \pmod{1}$, we know that $|-2z^3+3z^2| \leq 2$, and thus
\[ \left| \tfrac32 \epsilon(-2z^3+3z^2)\right| \leq 2 \cdot |\epsilon| \cdot 2 \leq 4r,\]
as desired.

\end{proof}
We are now ready to prove Theorem~\ref{thm:q2bdd}.
\begin{proof}[Proof of Theorem~\ref{thm:q2bdd}]
We will show that if $t \in D_{a,n}\cap \mathbb{Q}_2$ or if $t \in D_{b,n} \cap \mathbb{Q}_2$, then $f_t^n(1)$ will lie in either $\overline{D}\left( \frac{19}{2},2^{-4} \right)$ or $\overline{D} \left( \frac{27}{2}, 2^{-4}\right)$. Then, Lemma~\ref{lem:2cycle} will imply that the critical orbit is bounded. 

Let $n=3$. Then the first disk in each sequence is
\[ D_{a,3} = \overline{D}\left(321, 2^{-9}\right), \]
\[ D_{b,3} = \overline{D}\left(449, 2^{-9} \right).\]
First consider $t \in D_{a,3} \cap \mathbb{Q}_2$. In this case, we can show by direct computation that $f_t^3(1) \in \overline{D}\left(\tfrac{19}{2},2^{-4}\right)$:
\[ f_t(1) = 1-\tfrac32 t \equiv -\tfrac{961}{2} \pmod{2^8}\]
\[ \implies f_t^2(1) \equiv f_t\left(-\tfrac{961}{2}\right) \equiv \tfrac{111}{2} \pmod{2^6} \]
\[ \implies f_t^3(1) \equiv f_t \left( \tfrac{111}{2} \right) \equiv \tfrac{19}{2} \pmod{2^4}.\]

Similarly, we can show that if $t \in D_{b,3} \cap \mathbb{Q}_2$, then $f_t^3(1) \in \overline{D}\left(t\frac{27}{2},2^{-4}\right)$.
\[ f_t(1) = 1-\tfrac32 t \equiv -\tfrac{1345}{2} \pmod{2^8}\]
\[ \implies f_t^2(1) \equiv f_t\left( -\tfrac{1345}{2}  \right) \equiv \tfrac{79}{2} \pmod{2^6}\]
\[ \implies f_t^3(1) \equiv f_t \left(\tfrac{79}{2} \right) \equiv \tfrac{27}{2} \pmod{2^4}.\]

Next, let $n >3$. We will show that if $t \in D_{a,n}$ and $t_0 \in D_{b,n-1}$, then $f_t^3(1) \in \overline{D} \left(f_{t_0}^2(1), 2^{-(2n-2)}\right)$. Similarly, if $t \in D_{b,n}$ and $t_0 \in D_{a,n-1}$, then $f_t^3(1)\in \overline{D} \left(f_{t_0}^2(1), 2^{-(2n-2)}\right)$. From there, using Lemma~\ref{lem:tvst0} we can argue that $t$ is in the Mandelbrot set if $t_0$ is in the Mandelbrot set. The theorem will then follow by induction.

First, we track the first three iterates of $1$ under $f_t$ if $t \in D_{a,n}$. Let $t \equiv 1+5\cdot 2^{2n} \pmod{2^{2n+3}}$. Then,
\[ f_t(1) = 1-\tfrac32 t \equiv -\tfrac12 + 2^{2n-1} \pmod{2^{2n+2}}\]
\[ \implies f_t^2(1) \equiv f_t\left(-\tfrac12 + 2^{2n-1}\right) \equiv -\tfrac12 + 7 \cdot 2^{2n-3} \pmod{2^{2n}}\]
\[ \implies f_t^3(1) \equiv f_t\left( -\tfrac12 + 7 \cdot 2^{2n-3}\right) \equiv -\tfrac12 + 5 \cdot 2^{2n-5} \pmod{2^{2n-2}}.\]

We next do the same for $t \in D_{b,n}$, so $t \equiv 1+7\cdot2^{2n} \pmod{2^{2n+3}}$:
\[f_t(1) = 1-\tfrac32t \equiv -\tfrac12 + 3\cdot 2^{2n-1} \pmod{2^{2n+2}}\]
\[ \implies f_t^2(1) \equiv -\tfrac12 + 5 \cdot 2^{2n-3} \pmod{2^{2n}}\]
\[ \implies f_t^3(1) \equiv -\tfrac12+ 7 \cdot 2^{2n-5} \pmod{2^{2n-2}}.\]

Note by the above that if $t \in D_{a,n}$ and $t_0 \in D_{b,n-1}$, then $f_t^3(1) \equiv f_{t_0}^2(1) \pmod{2^{2n-2}}$. Note also that $|t - t_0| =2^{-(2n-2)}$.  Thus, we can repeatedly apply Lemma~\ref{lem:tvst0} to conclude that for $3 \leq m \leq n, |f_t^m(1) - f_{t_0}^{m-1}(1)| \leq 2^{-(2n-2m+4)}$. In particular, if $f_{t_0}^{n-1}(1)$ lies in either $\overline{D}\left(\frac{19}{2},2^{-4}\right)$ or $\overline{D}\left( \frac{27}{2},2^{-4}\right)$, then $f_{t}^n(1)$ will as well. Thus, by induction, the critical orbit is bounded.

The same argument works if $t \in D_{b,n}$ and $t_0 \in D_{a,n-1}$ by noting that $f_t^3(1) \equiv f_{t_0}^2(1) \pmod{2^{2n-2}}$.
\end{proof}

\begin{remark}
As Theorems \ref{thm: unbdd} and \ref{thm:q2bdd} are given in terms of closed $2$-adic discs, it is natural to consider an interpretation of these results in Berkovich space. For the interested reader, please see \cite{bakerrumely,benedetto2019dynamics,temkin2015introduction} for background on Berkovich space. Our results show that the type I point at $t=1$ is on the boundary of the slice of the degree 3 $2$-adic Mandelbrot set. Theorem \ref{thm: unbdd} shows that every type II point centered at $1+2^{2n-1}$ with radius $2^{-2n}$ for $n\geq 1$ is outside of the Mandelbrot set, along with all points below these type II points. Theorem \ref{thm:q2bdd} shows that for $n\geq 3$ every $\QQ_2$ point below a type II point centered at $1+5\cdot 2^{2n}$ or at $1+7\cdot 2^{2n}$ with radius $2^{-(2n+3)}$ is in the Mandelbrot set.

\end{remark}

\section{The Julia Set for $f_1$}\label{sec:julia}

In this section, we explore the filled Julia set for the following map in our family: 
\[ f_1(z) = -\tfrac{3}{2}(-2z^3+3z^2)+1 = 3z^3-\tfrac{9}{2}z^2+1.\]
In particular, we look in a neighborhood of the repelling fixed point $z=-\tfrac{1}{2}$, to examine the structural similarities with the Mandelbrot set near this parameter.

First we note that if $|z|>2$, then $|f_1(z)|=|z|^3$, and so $|f_1^n(z)| = |z|^{3^n}$ and the orbit of $z$ tends to infinity. Thus, the filled Julia set for $f_1$ is contained in the closed disk centered at $0$ of radius $2$. Moreover, the filled Julia set therefore consists of all points $z$ for which $\mathcal{O}_{f_1}(z)$ is contained inside the disk $\overline{D}(0,2)$. This bound for the radius of the filled Julia set is sharp, since $z=-\tfrac{1}{2}$ is a repelling fixed point, and therefore in the Julia set. Note that both critical points $0$ and $1$, along with $\tfrac{3}{2}$, are in the backward orbit of $-\frac{1}{2}$ and are thus in the Julia set as well. (We have $f_1(0) = 1, f_1\left(\tfrac{3}{2}\right) = 1$,and $f_1(1) = -\tfrac{1}{2}$.)

Similarly, if $1 < |z| < 2$ or if $2^{-1/2} <|z| < 1$, the orbit of $z$ will be unbounded: if $1<|z|<2$, then $|f_1(z)| = 2|z|^2 >2$, and if $2^{-1/2} <|z| < 1$, then $|f_1(z)| = 2|z|^2$, and $|f_1^2(z)| = 2(2|z|^2)^2 = 8|z|^4>2$. Thus the filled Julia set consists of points of absolute value 2, points of absolute value 1, and points of absolute value less than or equal to $2^{-1/2}$.

Since $z=-\frac{1}{2}$ is a fixed point with multiplier $|\lambda| = |f_1'\left(-\frac{1}{2}\right)| = |\frac{27}{4}| = 4$, we know it is repelling and therefore in the Julia set. This implies that there are points arbitrarily close to $-\frac{1}{2}$ with bounded orbits, and there are also points arbitrarily close to $-\frac{1}{2}$ with unbounded orbits. We explore a tree of points centered at $-\frac{1}{2}$ to determine if there are any patterns in how these points with bounded and unbounded orbits are distributed as we zoom in on our repelling fixed point.

\begin{figure}[h!]
\centering

\includegraphics[width=0.8\textwidth]{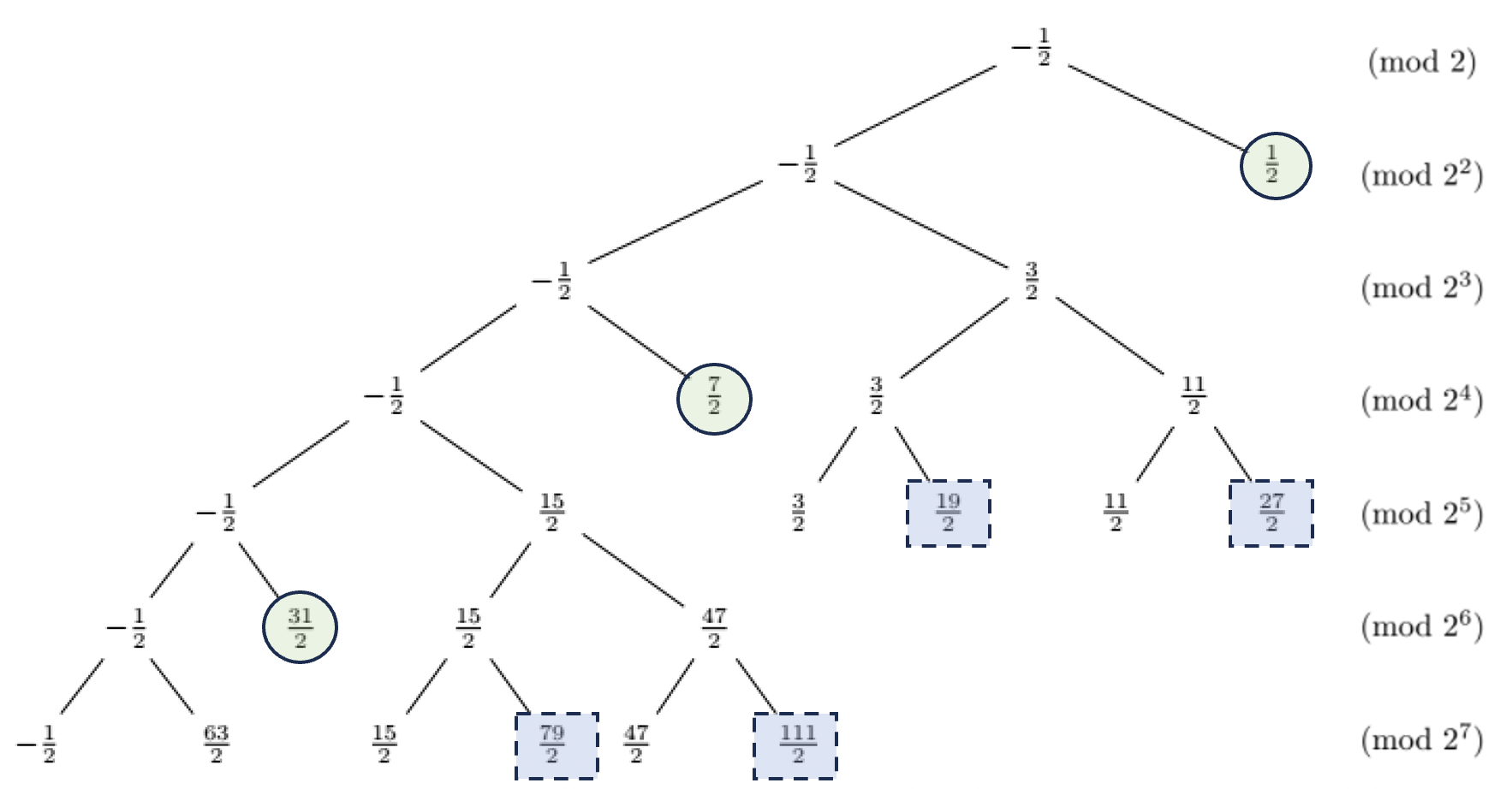}
\caption{}
\label{fig:Julia}
\end{figure}

\begin{figure}[h!]
\centering

\includegraphics[width=0.8\textwidth]{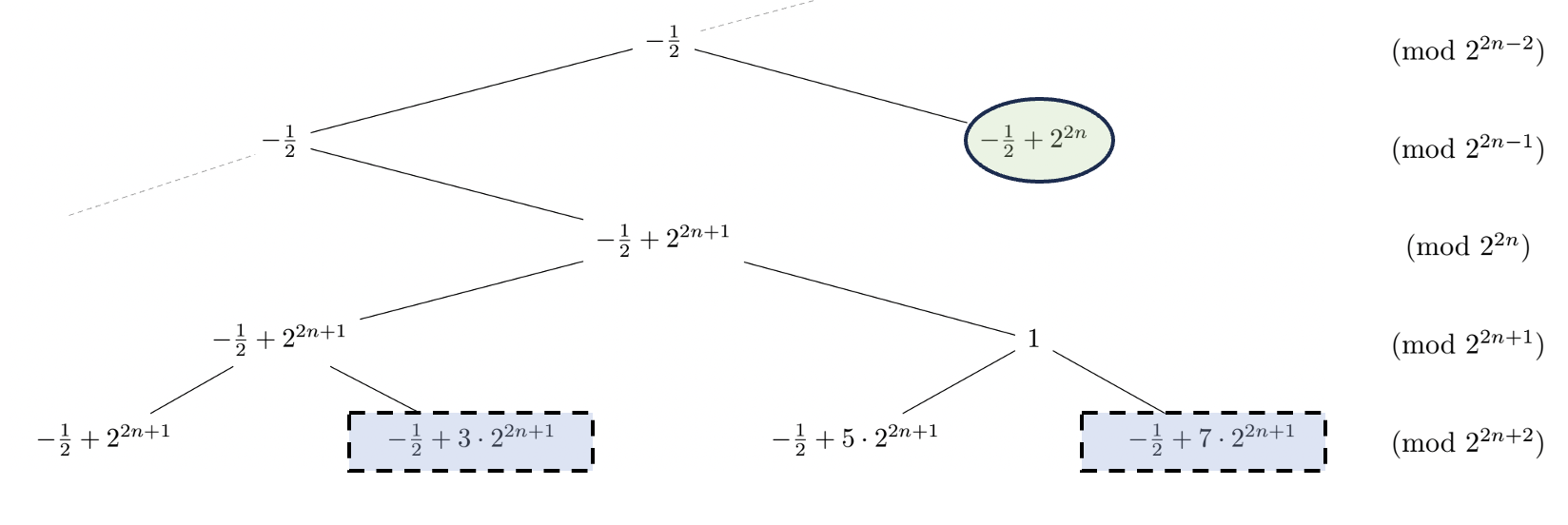}
\caption{}
    \label{fig:juliapattern}
\end{figure}

Propositions \ref{prop:unbddjulia} and \ref{prop:bddjulia} show that we have a tree pattern as shown in Figure \ref{fig:Julia} with the pattern shown in Figure \ref{fig:juliapattern} repeating every two levels as we zoom in on $-\frac12$. In Figure~\ref{fig:Julia} and Figure~\ref{fig:juliapattern}, a disk is labeled with a rectangle if all $\mathbb{Q}_2$ points in that disk belong to the filled Julia set, and with a circle if no points in that disk belong to the filled Julia set, i.e. if all points have unbounded orbits.

The following proposition gives a sequence of disks decreasing in radius and approaching $-\frac{1}{2}$ consisting entirely of points with unbounded orbits. It establishes a pattern of disks outside the filled Julia set for $f_1$ similar to the pattern of disks outside the Mandelbrot set in the parameter space that we observed in Theorem~\ref{thm: unbdd}.

\begin{prop}\label{prop:unbddjulia} All points contained in the disk $D_n = \overline{D}\left( -\frac{1}{2}+2^{2n}, 2^{-(2n+1)}\right)$ for $n \in \ZZ, n \geq 0$ have unbounded orbits under iteration of $f_1$.
\end{prop}

\begin{proof}
    First note that if $z \in D_0$, then $z = \frac12 +2k$ for some $k$ such that $|k| \leq 1$. Then, $|f_1(z)|>2$ and thus the orbit of $z$ is unbounded:
    \[ 
    \left|f_1\left(\tfrac12 + 2k\right)\right| = \left|3\left(\tfrac12 + 2k\right)^3-\tfrac{9}{2} \left(\tfrac12 + 2k\right)^2 + 1\right| = \left|\tfrac{1}{4}-\tfrac{9}{2}k+24k^3\right| = 4.
\]
    We complete the proof by noting that for $n>1$, $f_1(D_n) \subseteq D_{n-1}$. If $z \in D_n$ then $z = -\tfrac{1}{2} + 2^{2n} + 2^{2n+1}k$ for some $k$ such that $|k| \leq 1$. Then, computing $f_1(z)$, we get the following:
    \[ f_1\left( -\tfrac{1}{2} + 2^{2n}(1+2k) \right) = 3 \left(\ -\tfrac{1}{2} + 2^{2n}(1+2k) \right)^3-\tfrac{9}{2} \left(\ -\tfrac{1}{2} + 2^{2n}(1+2k) \right)^2 + 1\]
    \[ = -\tfrac{1}{2} + 27 \cdot 2^{2n-2}(1+2k)-9 \cdot 2^{4n}(1+2k)^2 + 3 \cdot 2^{6n}(1+2k)^3 \equiv -\tfrac{1}{2} + 2^{2n-2} \pmod{2^{2n-1}},\]
     and so it is contained in $D_{n-1}$.
    
    Since each $D_n$ maps into $D_{n-1}$, after $n$ iterations, every point in $D_n$ maps to a point of absolute value $4$ and thus has unbounded orbit.
\end{proof}

In the next proposition, we give two sequences of disks decreasing in radius and approaching $-\frac{1}{2}$ in which all of the $\QQ_2$ points have bounded orbits. This proposition gives a pattern of points with bounded orbits that is exactly analogous to the pattern we saw in the Mandelbrot set in Theorem~\ref{thm:q2bdd}. One can see the similarities between the Julia set for $f_1$ zoomed in to a particular level on $z=-\frac12$ and the slice of the Mandelbrot set for $f_t$ zoomed in to the same level on $t=1$ by comparing Figures~\ref{fig:repeatedmandelset} and~\ref{fig:juliapattern}. This similarity between the Mandelbrot set and the corresponding Julia set near a boundary point of the Mandelbrot set is analogous to what one sees for Misiurewicz points in the complex Mandelbrot set.

\begin{prop}\label{prop:bddjulia}
    Let $D_{a,n}$ and $D_{b,n}$ be sequences of disks defined as follows for $n \in \ZZ$, $n\geq 0$:
    \[ D_{a,n} = \overline{D}\left(-\tfrac12 + 5 \cdot 2^{2n+1}, 2^{-(2n+4)}\right),
    \]
    \[ D_{b,n} = \overline{D} \left(-\tfrac12 + 7 \cdot 2^{2n+1}, 2^{-(2n+4)}\right).\]
    
    Then, all of the $\QQ_2$ points in these disks have bounded orbits under iteration of $f_1$.
\end{prop}

\begin{proof}
    First consider $n=0$. Then, $D_{a,0} = \overline{D}\left(\frac{19}{2},2^{-4}\right)$ and $D_{b,0} = \overline{D}\left(\frac{27}{2},2^{-4}\right)$. We will show that the $\QQ_2$ points in $D_{b,0}$ map to $\QQ_2$ points in $D_{a,0}$ after two iterations of $f_1$, and that the $\QQ_2$ points in $D_{a,0}$ map into $D_{a,0}$ after two iterations of $f_1$. Thus, all of these points have bounded orbit.

    By explicitly computing the iteration of disks under $f_1$ below, we will see the following:
    \[ \overline{D}\left(\tfrac{27}{2},2^{-4}\right) \to  \overline{D}\left(2,2^{-2}\right) \to \overline{D}\left(3,2^{-2}\right) \leftarrow \overline{D}\left(\tfrac{19}{2},2^{-4}\right).\] 
    
    Moreover, we will show that the $\QQ_2$ points in $\overline{D}\left(3,2^{-2}\right)$ map into $\overline{D}\left(\tfrac{19}{2},2^{-4}\right)$. So, these $\QQ_2$ points simply move back and forth between these two disks every time we apply $f_1$, and thus their orbits stay bounded.

    Applying $f_1$ to a point of the form $\frac{19}{2}+16k, |k| \leq 1 $, we get the following:
    \[ f_1\left(\tfrac{19}{2}+16k\right) = 2167+11628k+20736k^2+12288k^3 \equiv 3 \pmod{4}.\]
    
    Doing the same with $z=\frac{27}{2} + 16k$,
    \[ f_1\left(\tfrac{27}{2} + 16k \right) = 6562 + 24300k+29952k^2+12288k^3 \equiv 2 \pmod{4}.\]
    
    Next, with $z=2+4k$,
    \[ f_1 (2+4k) = 7+72k+216k^2+192 k^3 \equiv 3 \pmod{4}.\]  
    
    Finally, applying $f_1$ to a point of the form $3+4k, k \in \QQ_2$, we get the following:
    \[ f_1(3+4k) = \tfrac{83}{2} + 216k +360 k^2 + 192k^3 \equiv \tfrac{19}{2}+8k+8k^2 \pmod{16}.\]
    
    Note that for $k \in \QQ_2$, we have $8k+8k^2 = 8k(k+1) \equiv 0 \pmod{16}$, and so $f_1(3+4k)$ lies in the disk $\overline{D}\left(\tfrac{19}{2},2^{-4}\right)$.

    For $n \geq 1$, we will show that if $z \in D_{a,n}$, then $f_1(z) \in D_{b,n-1}$, and if $z \in D_{b,n}$, then $f_1(z) \in D_{a,n-1}$. This will complete the proof that all $\QQ_2$ points in these disks have bounded orbits. First, suppose $z \in D_{a,n}$. Then we can write $z=-\tfrac{1}{2}+2^{2n+1}(5+8k)$, where $|k| \leq 1$. Applying $f_1$ to $z$, we obtain:
    \[ f_1 \left(-\tfrac{1}{2}+2^{2n+1}(5+8k) \right) = -\tfrac{1}{2}+27\cdot 2^{2n-1}(5+8k) - 9 \cdot 2^{4n+2} (5+8k)^2+3 \cdot 2^{6n+3} (5+8k)^3 \]
    \[ \equiv -\tfrac12 + 7 \cdot 2^{2n-1} \pmod{2^{2n+2}} \in D_{b,n-1}.\]

    Next, consider $z \in D_{b,n}$. Then we can write $z=-\tfrac12 + 2^{2n+1}(7+8k)$, where $|k| \leq 1$. Applying $f_1$ to $z$, we get the following:
    \[ f_1\left(-\tfrac12 +  2^{2n+1}(7+8k) \right) = -\tfrac12 + 27 \cdot 2^{2n-1} (7+8k) - 9 \cdot2^{4n+2}(7+8k)^2+3\cdot2^{6n+3}(7+8k)^3
    \]
    \[\equiv -\tfrac12 + 5 \cdot 2^{2n-1} \pmod{2^{2n+2}} \in D_{a,n-1}.\]
Thus, for all $n \geq 0$, all $\mathbb{Q}_2$ points in $D_{a,n}$ or $D_{b,n}$ eventually map to $\mathbb{Q}_2$ points in $\bar{D}\left( \frac{19}{2}, 2^{-4} \right)$, and therefore have bounded orbits.
\end{proof}

In the proposition above, the $\QQ_2$ points in the disks $\overline{D}(3,2^{-2})$ and $\overline{D}\left( \frac{19}{2}, 2^{-4}\right)$ map to each other under $f_1$. This is because there is a (repelling) $2$-cycle in these disks, namely the roots of $6z^2-3z-1$. The roots of this polynomial, call them $\alpha_1$ and $\alpha_2$, map to each other under $f_1$. A Newton polygon analysis shows that one of these points, say $\alpha_1$, has absolute value 1 and the other has absolute value 2. Using Hensel's Lemma, we can see that $\alpha_1 \in \mathbb{Z}_2$ with $\alpha_1 \equiv 7 \pmod{16}$. So, the multiplier for this cycle is 
\[ |\lambda| = |f_1'(\alpha_1)f_1'(\alpha_2)| = |\alpha_1||\alpha_1-1||\alpha_2||\alpha_2-1| = 1 \cdot \tfrac12 \cdot 2 \cdot 2 = 2.\]

Since this cycle is repelling, there will be $\CC_2$ points that are not in $\QQ_2$ in these disks whose orbits will escape to infinity, while every $\QQ_2$ point in these disks has a bounded orbit. This is precisely the same as what we observed in Theorem~\ref{thm:q2bdd}.
 \nocite{*}
%\bibliographystyle{plain}
%\bibliography{mandelbrot}

\end{document}